\documentclass[a4paper,numbook,babel,final,envcountsame,11pt]{article}

\usepackage{amssymb}
\usepackage{amsthm}
\usepackage{amsmath}
\usepackage{graphicx}
\usepackage{array,hhline}

\newtheorem{lemma}{Lemma}[section]
\newtheorem{theorem}[lemma]{Theorem}
\newtheorem{proposition}[lemma]{Proposition}

\newtheorem{corollary}[lemma]{Corollary}
\theoremstyle{definition}
\newtheorem{definition}[lemma]{Definition}
\newtheorem{remark}[lemma]{Remark}

\numberwithin{equation}{section}
\numberwithin{figure}{section}

\newcommand{\Xset}{\mathcal{X}}
\newcommand{\Yset}{\mathcal{Y}}
\newcommand{\Zset}{\mathcal{Z}}

\newcommand{\eg}{\emph{e.g.}}

\begin{document}
\date{}
\title{On the intersection points of two plane algebraic curves}
\author{H. Hakopian, D. Voskanyan \\
} \maketitle

\begin{abstract}We  prove that a set $\mathcal X\subset \mathbb{C}^2,\ \#{\mathcal X}=mn,\  m\le n, $ is the set of intersection points of some two plane algebraic curves
of degrees $m$ and $n,$ respectively, if and only
if the following conditions are satisfied:\\
a) Any curve of degree $m+n-3$ containing all but one point of  \indent $\mathcal X$, contains
 all of
$\mathcal X,$\\
b) No curve of degree less than $m$ contains all of $\mathcal X.$

Let us mention that the  conditions a) and b) in the``only if" direction of this result follow from the Ceyley-Bacharach and Noether theorems, respectively.
\end{abstract}


\section{Introduction\label{Sec1}, $n$-independence\label{Sec2}}

Let $\Pi_n$ be the space of all polynomials in two variables and of total
degree less than or equal to $n$. Its dimension is given by
\begin{equation*}\label{eq:dim}
  N := \dim \Pi_n = \binom{n+2}{2}.
\end{equation*}
A plane algebraic curve is the zero set of some bivariate
polynomial. To simplify notation, we shall use the same letter $p$,
say, to denote the polynomial $p$ and the curve given by the
equation $p(x,y)=0$. More precisely, suppose $p$ is a polynomial
without multiple factors. Then the plane curve defined by the
equation $p(x,y)=0$ shall also be denoted by $p$.

So lines, conics, and cubics are equivalent to polynomials of degree
$1$, $2$, and $3$, respectively; a reducible conic is a pair of
lines, and a reducible cubic is a triple of lines, or consists of a
line and an irreducible conic.

The following is a Linear Algebra fact:
\begin{lemma} \label{curve} For any $N-1=(1/2)k(k+3)$ points in the plane there is a curve of degree $k$ passing through them.
\end{lemma}

Suppose a set of $k$ distinct points is given:
\begin{equation*}\label{eq:knset}{\mathcal X}_k =
\{(x_i,y_i): i=1,2,\ldots,k\}\subset \mathbb{C}^2.\end{equation*}
The problem of finding a polynomial $p \in \Pi_n$ which satisfies
the conditions
\begin{equation}\label{eq:intpr}
  p(x_i,y_i) = c_i, \quad i=1, \ldots, k,
\end{equation}
is called \emph{interpolation problem}.
\begin{definition}
The set of points ${\mathcal X}_k$ is called  \emph{$n$-poised}, if for any data $(c_1,\ldots, c_m)$,
there is a \emph{unique} polynomial $p\in\Pi_n$ satisfying the
conditions \eqref{eq:intpr}.
\end{definition}
\noindent By a Linear Algebra argument a necessary condition for
$n$-poisedness is
\begin{equation*}
  k = \#{\mathcal X}_k = \dim \Pi_n = N .
\end{equation*}
 A polynomial $p\in\Pi_n$ is called
\emph{$n$-fundamental polynomial} of a point
$A\in{\mathcal X},$  if
\begin{equation*}\label{eq:fundpol}
  p(A) = 1 \qquad\text{and}\qquad  p\big\vert_{{\mathcal X}\setminus\{A\}} = 0,
\end{equation*}
where $p\big\vert_{\mathcal X}$ means the restriction of $p$ to
${\mathcal X}.$
We shall denote such a polynomial by $p_{A,{\mathcal X}}^\star.$

\noindent Sometimes we call $n$-fundamental also a polynomial from $\Pi_n$
that just vanishes at all the points of ${\mathcal X}$ but $A,$ since such
a polynomial is a nonzero constant multiple of $p_A^\star.$
A fundamental polynomial can be described as a plane curve containing all but one point of $\Xset.$

Next we consider an important concept of $n$-independence and $n$-dependence
of point sets (see \cite{EGH96}, \cite{HJZ09a} - \cite{HM13}).
\begin{definition}
A set of points ${\mathcal X}$ is called
\emph{$n$-independent}, if each its point
has an $n$-fundamental polynomial. Otherwise, it is called \emph{$n$-dependent}.
\end{definition}
\begin{definition}\label{def:depset}
 A set of points ${\mathcal X}$ is called
\emph{essentially $n$-dependent}, if none of its points has an $n$-fundamental polynomial.
\end{definition}

If a point set ${\mathcal X}$ is $n$-dependent, then for some $A\in{\mathcal X}$, there is
 no $n$-fundamental polynomial, which means
that for any polynomial $p\in\Pi_n$ we have that
\begin{equation*}
  p\big\vert_{{\mathcal X}\setminus\{A\}} = 0
  \quad \implies \quad p(A)=0.
\end{equation*}

Thus a set $\mathcal X$ is essentially $k$-dependent means that any
plane curve of degree $k$ containing all but one point of $\mathcal X$, contains
all of $\mathcal X.$

Since fundamental polynomials are linearly independent we obtain that a
necessary condition for $n$-independence is
\begin{equation*}
  \#{\mathcal X} \le \dim \Pi_n = N.
\end{equation*}

\noindent It is easily seen that $n$-independence of
${\mathcal X}_k$ is equivalent to the \emph{solvability} of the
interpolation problem \eqref{eq:intpr}, meaning that for any data
$\{c_1,\ldots, c_k\}$ there exists a (not necessarily unique)
polynomial $p\in \Pi_n$ satisfying the interpolation conditions
\eqref{eq:intpr}.
In the case $k=N,$ i.e., for a point set ${\mathcal X}_N,$
the $n$-independence means $n$-poisedness.

We have the following
\begin{proposition}[ \cite{HM12}, Lemma 2.2]\label{cor:ind1}
Suppose that the point set ${\mathcal X}$ is
$n$-independent and each point of the set ${\mathcal Y}$ has $n$-fundamental
polynomial with respect to the set ${\mathcal X}\cup {\mathcal
Y}.$ Then the latter point set is
$n$-independent, too.
\end{proposition}

\begin{corollary}[ \cite{HM12}, Prop. 2.3]\label{cor:ind1}
Suppose that a curve $\sigma_k,$  of degree $k$ contains an $n$-independent point set ${\mathcal X}.$
Suppose also that a set ${\mathcal Y}$ is outside of $\sigma_k$ and is $(n-k)$-independent. Then the set ${\mathcal X}\cup {\mathcal
Y}$ is
$n$-independent.
\end{corollary}

Below we give a characterization of $n$-dependence of point sets consisting of at most $3n$ points.

\begin{theorem}[\cite{HM12}, Thm.~5.1]\label{thm:3n}
A set $\Xset$ consisting of at most $3n$ points is $n$-dependent if
and only if one of the following holds.

a) $n+2$ points are collinear,

b) $2n+2$ points belong to a (possibly reducible) conic,

c) $\#\Xset = 3n,$ and there exist $\sigma_3\in\Pi_3$ and $\sigma_n\in\Pi_n$
such that $\Xset = \sigma_3 \cap \sigma_n.$
\end{theorem}

Next we bring three corollaries of this result.
\begin{corollary}\label{cor:2n+1}
A set $\Xset$ consisting of at most $2n+1$ points is $n$-dependent if
and only if $n+2$ points are collinear.
\end{corollary}

\noindent
A generalization of this result allowing for multiple points can be
found in (\cite{H00}, Theorem 9). From Corollary \ref{cor:2n+1} we get immediately the following result of Severi \cite{S}:

\begin{corollary}[Severi, \cite{S}]\label{cor:n+1}
Any set $\Xset$ consisting of at most $n+1$ points is $n$-independent.
\end{corollary}

\begin{corollary}\label{cor:3n-1}
A set $\Xset$ consisting of at most $3n-1$ points is $n$-dependent if
and only if one of the following holds.
\vspace{-2mm}
\begin{enumerate}
\setlength{\itemsep}{0mm}
\item
$n+2$ points are collinear,
\item
$2n+2$ points belong to a (possibly reducible) conic.
\end{enumerate}
\end{corollary}
A special case of above result, when $\#\Xset\le 2n+2,$ can be found in (\cite{EGH96}, Prop. 1).

\begin{lemma}\label{red} Assume that $\mathcal X$ is essentially $k$-dependent and  $\sigma_n$ is a curve of degree $n.$
Assume also that the point set $\mathcal Y:=\mathcal X\setminus \sigma_n$ is not empty. Then $\mathcal Y$ is essentially $(k-n)$-dependent.
\end{lemma}
Indeed, assume conversely that a point $A\in {\mathcal Y}$ has a $(k-n)$-fundamental polynomial: $p_{A,\mathcal Y}^\star.$ Then it is easily seen that the polynomial $p:=p_{A,\mathcal Y}^\star\sigma_{n}$ is a $k$-fundamental polynomial of $A$ in the set $\mathcal X,$
which is a contradiction.

Assume that a curve $\sigma_n$ of degree $n$ is reducible, i.e.,
\begin{equation}\label{comp}\sigma_n=\sigma_{n_1}\cdots\sigma_{n_s},\end{equation}
 where the component $\sigma_{n_i}$ has degree $n_i.$

Denote by  ${\mathcal X}_i,\ i=1,\ldots,s,$ the set of points from ${\mathcal X}\cap\sigma_{n_i}$ which do not lay in other components $\sigma_{n_j},\ j\neq i,$ i.e.,
 \begin{equation}\label{comp'}{\mathcal X}_i={\mathcal X}\setminus\big(\bigcup_{{j=1}\atop {j\neq i}}^s\sigma_{n_j}\big).\end{equation}

\noindent We call a component $\sigma_{n_i}$\emph{not empty} with respect to the set ${\mathcal X}$ if ${\mathcal X}_i\neq  \emptyset.$

\begin{lemma}\label{reduc} Assume that $\mathcal X\subset \sigma_n$ is essentially $k$-dependent, where the curve $\sigma_n$ of degree $n$ is reducible,
given by \eqref{comp}. Assume also that all the components are not empty with respect to $\mathcal X.$\\
Then each set ${\mathcal X}_i$ given in \eqref{comp'} is essentially $(k-n+n_i)$-dependent.
\end{lemma}
Indeed assume that for some $i,\ 1\le s\le s$ the set  ${\mathcal X}_i$ is not essentially $(k-n+n_i)$-dependent, i.e., a point $A\in {\mathcal X}_i$ has a $(k-n+n_i)$-fundamental polynomial: $p_i.$ Then it is easily seen that the polynomial $p:=p_i\prod_{{j=1}\atop {j\neq i}}^s\sigma_{n_j}$ is a $k$-fundamental polynomial of $A$ in the set $\mathcal X,$
which is a contradiction.

\subsection{The completeness of point sets in plane curves}

\begin{definition}\label{def:depset}
Let $\sigma_k$ be a plane curve of degree $k$, without multiple components. Then the point set ${\mathcal X}\subset \sigma_k$ is called
\emph{$n$-complete in} $\sigma_k,$ if the following assertion holds:
$$p\in \Pi_n,\ p\big\vert_{\mathcal X}=0\ \Rightarrow \ p=q\sigma_k,\ q\in \Pi_{n-k}.$$
\end{definition}
The $n$-completeness in the case $k>n$ means  that $p\in \Pi_n,\ p\big\vert_{\mathcal X}=0\ \Rightarrow \ p=0.$ Therefore we have
\begin{lemma} \label{compoised} Let $k>n.$ Then a set of points $\mathcal X \subset \sigma_k$ is $n$-complete in $\sigma_k,$ if and only if $\mathcal X$ contains an $n$-poised subset $\mathcal X_0.$
\end{lemma}

Consider the following two linear spaces of polynomials:
\begin{equation*}{{\mathcal P}}_{n,{\mathcal X}}:=\left\{p\in \Pi_n
: p\big\vert_{\mathcal X}=0\right\},\ \ \ {{\mathcal P}}_{n,{\sigma_k}}:=\left\{p\sigma_k: p\in \Pi_{n-k}\right\},
\end{equation*}
where $\Xset$ is a point set and $\sigma_k\in\Pi_k.$
Note that \begin{equation}\label{d}{{\mathcal P}}_{n,{\mathcal X}}\supset{{\mathcal P}}_{n,{\sigma_k}}\ \hbox{if}\ {\mathcal X}\subset \sigma_k.
\end{equation}

Then we have also that \begin{equation}\label{dd}{{\mathcal P}}_{n,{\mathcal X}}={{\mathcal P}}_{n,{\sigma_k}}\iff {\mathcal X}\subset \sigma_k\ \hbox{is
$n$-complete in}\ \sigma_k.\end{equation}

Now, we readily conclude from \eqref{d} and \eqref{dd} that
\begin{equation}\label{d'}\dim{{\mathcal P}}_{n,{\mathcal X}}=\dim{{\mathcal P}}_{n,{\sigma_k}}\iff {\mathcal X}\subset \sigma_k\ \hbox{is
$n$-complete in}\ \sigma_k.\end{equation}

Evidently we have that \begin{equation}\label{dd'}\dim{{\mathcal P}}_{n,{\sigma_k}}=\dim\Pi_{n-k}.\end{equation}

For $\dim{\mathcal P}_{n,{\mathcal X}}$ we have the following well-known (see e.g. \cite{HJZ09a}, \cite{HM12})

\begin{proposition} \label{cor:ind4} Let ${\mathcal X}_0$ be a maximal $n$-independent subset of ${\mathcal X},$ i.e.,  ${\mathcal X}_0$ is $n$-independent and
${\mathcal X}_0\cup \{A\}$ is $n$-dependent for any $A\in {\mathcal X}\setminus {\mathcal X}_0.$ Then we have that
\begin{equation}\label{eq:theta2}\dim{{\mathcal P}}_{n,{\mathcal X}}=\dim {{\mathcal P}}_{n,{\mathcal X}_0}=\dim\Pi_n-\#\mathcal X_0.
\end{equation}
\end{proposition}
Set $$d(k,n):=\dim\Pi_n-\dim\Pi_{n-k}.$$
It is easily seen that $d(k,n)=(n+1)+n+\cdots+(n-k+2)=\frac{1}{2}k(2n-k+3),$ if $k\le n+2.$

Finally, in view of \eqref{dd}-\eqref{dd'}, we get the following simple criterium for the completeness:
\begin{proposition}[\eg, \cite{R11}, Prop.~3.1]\label{prop:complete}
Let $k\le n+2$ and $\sigma_k$ be a plane curve of degree $k.$ Then the point set ${\mathcal X}\subset \sigma_k$ is $n$-complete in $\sigma_k,$ if and only if
${\mathcal X}$ contains an $n$-independent subset ${\mathcal X}_0\subset{\mathcal X}$ with $\#{\mathcal X}_0=d(k,n).$
\end{proposition}

Note that in the cases $k=n+1, n+2,$ we have that $d(k,n)=\dim\Pi_k,$ and Proposition follows from Lemma \ref{compoised}.

\begin{theorem}[Ceyley-Bacharach] \label{CB} Suppose that a set $\mathcal X,\ \#{\mathcal X}=mn,$ is the set of intersection points of some two plane curves $\sigma_m$ and
$\sigma_n,$ of degrees $m$ and $n,$ respectively: $\mathcal X =\sigma_m\cap \sigma_n.$ Then we have that

a)  the set $\mathcal X$ is  essentially $\kappa$-dependent;

b)  the set $\mathcal X$ is $(\kappa+1)$-independent;

c)  for any point $A\in \mathcal X$ the point set $\mathcal X\setminus \{A\}$ is $\kappa$-independent.
\end{theorem}

\begin{theorem}[Noether] \label{N} Suppose that a set $\mathcal X,\ \#{\mathcal X}=mn,$ is the set of intersection points of some two plane curves $\sigma_m$ and
$\sigma_n,$ of degrees $m$ and $n,$ respectively: $\mathcal X =\sigma_m\cap \sigma_n.$ Then for any polynomial $p_k\in\Pi_k,$ vanishing on $\Xset,$ we have that
\begin{equation}\label{NT}p_k=A_{k-m} \sigma_m+B_{k-n} \sigma_n,\end{equation}
where $A_{k-m}\in\Pi_{k-m}$ and $B_{k-n}\in\Pi_{k-n}.$
\end{theorem}

\begin{corollary} \label{cN} Suppose that a set $\mathcal X,\ \#{\mathcal X}=mn,\ m\le n,$ is the set of intersection points of some two plane curves $\sigma_m$ and
$\sigma_n,$ of degrees $m$ and $n,$ respectively: $\mathcal X =\sigma_m\cap \sigma_n.$ Then no curve of degree less than $m$ contains all of $\mathcal X.$
\end{corollary}
Indeed, suppose conversely that a curve $p$ of degree less than $m$ contains all of $\mathcal X.$
Then we get from \eqref{NT} that $p=0,$ which is a contradiction.

\section{Main results}

 Throughout this section let us set $$\kappa:=\kappa(m,n):=m+n-3.$$

\begin{theorem}\label{main}
A set $\mathcal X$ with $\#{\mathcal X}=mn,\ m\le n,$ is the set of intersection points of some two plane curves
of degrees $m$ and $n,$ respectively, if and only
if the following conditions are satisfied:\\
a) Any plane curve of degree $\kappa$ containing all but one point of $\mathcal X$, contains
all of $\mathcal X,$\\
b) No curve of degree less than $m$ contains all of $\mathcal X.$
\end{theorem}

Let us mention that the necessity of the conditions a) and b) follow from Theorem \ref{CB} and Corollary \ref{cN}, respectively.
Note also that the condition a) above means that the point set $\mathcal X$ is essentially $\kappa$-dependent, while the condition b) means that the set $\mathcal X$ contains an $(m-1)$-poised set.

Next we prove the part of sufficience in the cases $m=1,2,3.$

\subsection{The proof of Thorem \ref{main} in the cases $m=1,2,3$}

\emph{The case $m=1.$}\hskip.01cm

\noindent In this case an essentially $(n-2)$-dependent set $\mathcal X=\{A_1,\ldots,A_n\}$ is given. By  Corollary \ref{cor:2n+1} we get that the $n$ points are collinear, i.e., belong to a line $\sigma_1.$
Hence we get that $\mathcal X=\sigma_1\cap\sigma_n,$ where $\sigma_n$ has $n$ line components intersecting $\sigma_1$ at the $n$ points of $\mathcal X$, respectively.

\noindent \emph{The case $m=2.$} \hskip.01cm

\noindent In this case an essentially $(n-1)$-dependent set $\mathcal X=\{A_1,\ldots,A_{2n}\}$ is given. By  Corollary \ref{cor:3n-1} we get that either $n+1$ points of $\mathcal X$ are collinear, i.e., belong to a line $\sigma_1,$ or all $2n$ points of $\mathcal X$ belong to a conic $\sigma_2.$
Suppose first that $n+1$ points of $\mathcal X$ belong to a line $\sigma_1.$ Then by denoting $\mathcal Y=\mathcal X\setminus \sigma_1$ we have that $\#\mathcal Y\le n-1.$ By the condition b) we have that  $\mathcal Y\neq \emptyset.$ Now we get from Lemma \ref{red} that the set $\mathcal Y$ is essentially $(n-2)$-dependent, which contradicts Corollary \ref{cor:n+1}.

Next, suppose that all the $2n$ points of $\mathcal X$ belong to the conic $\sigma_2.$ First consider the case when the conic $\sigma_2$ is irreducible.
Then we get that $\mathcal X=\sigma_2\cap\sigma_n,$ where $\sigma_n$ has $n$ line components intersecting $\sigma_2$ at $n$ disjoint couples of points, respectively.

Finally suppose that the conic $\sigma_2$ is reducible, i.e., is a pair of lines: $\sigma_2=\sigma_1\sigma_1'$. First let us prove that each of the component lines contains exactly $n$ points from $\Xset$ hence the intersection point of $\sigma_1$ and $\sigma_1'$ does not belong to $\Xset.$ Assume conversely that a component, say $\sigma_1,$ contains $n+1$ points from $\Xset.$
Then by denoting $\mathcal Y=\mathcal X\setminus \sigma_1$ we have that $\#\mathcal Y\le n-1.$ By the condition b) we have that $\mathcal Y\neq \emptyset.$ Now we get from Lemma \ref{red} that the set $\mathcal Y$ is essentially $(n-2)$-dependent, which contradicts Corollary \ref{cor:n+1}.

Hence we get that $\mathcal X=\sigma_2\cap\sigma_n,$ where $\sigma_n$ has $n$ line components intersecting $\sigma_1$ and $\sigma_1'$ at $n$ disjoint couples of points, one from $\sigma_1$ and another from $\sigma_1'.$

\noindent \emph{The case $m=3.$} \hskip.01cm

 In this case an essentially $n$-dependent set $\mathcal X=\{A_1,\ldots,A_{3n}\}$ is given. By  Theorem \ref{thm:3n} we get that either $n+2$ points  belong to a line $\sigma_1,$  $2n+2$ points belong to a conic $\sigma_2,$ or $\mathcal X=\sigma_3\cap\sigma_n,$ where $\sigma_i\in\Pi_i.$ It is enough to exclude here the first two possibilities.

Suppose first that $n+2$ points of $\mathcal X$ belong to a line $\sigma_1.$ Then by denoting $\mathcal Y=\mathcal X\setminus \sigma_1$ we have that $\#\mathcal Y\le 2n-2.$ By the condition b) we have that  $\mathcal Y\neq \emptyset.$ Now we get from Lemma \ref{red} that the set $\mathcal Y$ is essentially $(n-1)$-dependent. Hence by Corollary \ref{cor:2n+1} we get that $n+1$ points of $\Yset$ belong to a line $\sigma_1'.$ Then by denoting $\mathcal Z=\mathcal X\setminus (\sigma_1\cup\sigma_1')$ we have that $\#\mathcal Z\le n-3.$ By the condition b) we have that $\mathcal Z\neq \emptyset.$ Now we get from Lemma \ref{red} that the set $\mathcal Z$ is essentially $(n-2)$-dependent, which contradicts Corollary \ref{cor:n+1}.

Next, suppose that $2n$ points belong to a conic $\sigma_2.$ Then by denoting $\mathcal Y=\mathcal X\setminus \sigma_2$ we have that  $\#\mathcal Y \le n, \mathcal Y \neq \emptyset.$ Now we get from Lemma \ref{red} that the set $\mathcal Y$ is essentially $(n-2)$-dependent. Hence by By  Corollary \ref{cor:2n+1} we get that the set $\Yset$ has exactly $n$ collinear points, belonging to a line $\sigma_1.$   Then by denoting $\mathcal Z=\mathcal X\setminus \sigma_1$ we have that  $\#\mathcal Z \le 2n, \mathcal Z \neq \emptyset.$      Now we get readily from Lemma \ref{red} that the set $\mathcal Z$ is essentially $(n-1)$-dependent. Next we conclude, as above, from Corollary 1.4, that $\Zset$ contains exactly $2n$ points and in the case when the conic $\sigma_2$ has two line components, then each of the component lines contains exactly $n$ points from $\Xset.$

Now, from Proposition \ref{prop:complete} we get that $\mathcal X$ is not $n$-complete in $\sigma_3:=\sigma_1\sigma_2.$ Hence there is a polynomial $\sigma_n\in\Pi_n$ vanishing at $\Xset$ but not on whole $\sigma_3,$ in particular
$$\Xset\subset \sigma_3\cap\sigma_n.$$
It remains to verify that $\sigma_3$ and $\sigma_n$ have no common components.

Indeed, suppose that the  common component of the highest degree is $\sigma,$ where $\sigma\in\Pi_2.$
Then we have that $\sigma_n=\sigma\sigma_{n-2},$ where $\sigma_{n-2}\in\Pi_{n-2}.$ Now consider the line component $\sigma_1,$ of  $\sigma_3$ which is not a component of $\sigma.$
On that component we have exactly $n$ points which are outside of $\sigma.$ Hence these $n$ points belong to the curves $\sigma_1$ and $\sigma_{n-1},$ which contradicts to the Bezout theorem, since  the curves have no common component.

Finally, suppose that $\sigma\in\Pi_1.$
Then we have that $\sigma_n=\sigma\sigma_{n-1},$ where $\sigma_{n-1}\in\Pi_{n-1}.$ Now consider a (the) component $\sigma_k,\ k\le 2$ of  $\sigma_3$ different from $\sigma.$
On that component we have exactly $kn$ points which are outside of $\sigma.$ Hence these $kn$ points belong to the curves $\sigma_k$ and $\sigma_{n-1},$ which contradicts to the Bezout theorem, since  the curves have no common component.

\subsection{The proof of Thorem \ref{main} in the case $m\ge 4.$}

The proof of the sufficiency part of Thorem \ref{main} is completed in the forthcoming Theorem \ref{6} at the end of the section.

Let us start the discussion with the following

\begin{theorem} \label{1} Suppose that an irreducible curve  $\sigma_m$ of degree $m$ contains a set $\mathcal X$ of $mn$ points.  Then the following statements hold:\\
(a) If the set $\mathcal X$ is $\kappa$-independent then it is $n$-complete in $\sigma_m.$\\
(b) Suppose that $3\le m\le n+2.$ If the set $\mathcal X$ is $n$-complete in $\sigma_m$  then  it is $\kappa$-independent.\\
\end{theorem}
\begin{proof} $\ $ Part a): $\ $ Suppose that a set $\mathcal X\subset\sigma_m$ is not $n$-complete in $\sigma_m.$
Then there is a polynomial $\sigma_n\in\Pi_n$   that vanishes on $\mathcal X$ but not on $\sigma_m.$
Then, since the curve $\sigma_m$ is irreducible we conclude from the Bezout theorem that
$$\mathcal X =\sigma_m\cap \sigma_n.$$
Now we get from Theorem \ref{CB}, a), that $\mathcal X$ is $\kappa$-dependent.
More precisely, we get from Theorem \ref{CB}, a), that $\mathcal X$ is essentially $\kappa$-dependent and, from the item c), that for any point $A\in \mathcal X$ the point set $\mathcal X\setminus \{A\}$ is $\kappa$-independent.

\indent Part b): $\ $ Suppose that the set of points $\mathcal X$ is $n$-complete in $\sigma_m.$
Then, according to Proposition \ref{prop:complete}, we have that $\mathcal X$ contains an $n$-independent subset $\mathcal Y$ of $d(m,n)$ points. Since $m\le n+2$ the number of points in $\mathcal Z:=\mathcal X\setminus \mathcal Y$ equals
$$mn-d(m,n)=mn - \frac{1}{2}m(2n-m+3)=\frac{1}{2}m(m-3).$$
Thus in the case $m=3$ we have that $\mathcal X=\mathcal Y$ is $\kappa=n$-independent.
Now assume that $m>3.$
In view of Lemma \ref{curve} we have that there is a curve $\sigma_{m-3}$ of degree $m-3$ containing all the points of $\mathcal Z.$
Denote by $\bar{\mathcal Z}$ the set of all points of $\mathcal X$ in $\sigma_{m-3}.$ Since the curve $\sigma_m$ is irreducible it has no common component with  $\sigma_{m-3}.$
Next, we have that
$$\bar{\mathcal Z}\subset \sigma_{m} \cap \sigma_{m-3}.$$
Therefore by Theorem \ref{CB}, b), the set $\bar{\mathcal Z}$ is $m+(m-3)-2=(2m-5)$-independent.
On the other hand we have that $\kappa=m+n-3\ge 2m-5.$ Therefore
the set $\bar{\mathcal Z}$ is $\kappa$-independent.
Then, we have that the set $\mathcal X\setminus \bar{\mathcal Z}\subset \mathcal X\setminus \mathcal Z=\mathcal Y$ is $n$-independent.
By Corollary \ref{cor:ind1} the set $\mathcal X$ is $\kappa$-independent.
\end{proof}

We get immediately from the proof of the part a) (the last sentence there):
\begin{corollary} \label{1'}
Suppose that an irreducible curve  $\sigma_m$ of degree $m$ contains a set $\mathcal X$ of $mn$ points, which is not $n$-complete.
Then the set $\mathcal X$ is essentially $\kappa$-dependent and for any point $A\in \mathcal X$ the set $\mathcal X\setminus \{A\}$ is $\kappa$-independent.
\end{corollary}
We get from the proof of the part b) of Theorem \ref{1} the following
\begin{proposition}  \label{1''} Suppose that $3\le m\le n+2$ and a (not necessarily irreducible) curve  $\sigma_m$ of degree $m$ contains a set $\mathcal X$ of $\le mn$ points, which is $n$-complete.
Then the set $\mathcal X$ is not essentially $\kappa$-dependent.
\end{proposition}
\begin{proof}By proof of part b) of Theorem \ref{1} we have that there is a curve $\sigma_{m-3}$ of degree $m-3$ passing through all the points of the set $\mathcal Z=\mathcal X \setminus \mathcal Y.$ In the case $m=3$ we have that $\mathcal X = \mathcal Y$ and thus is $\kappa=n$-independent. Now suppose that $m>3.$
Let us show that $\sigma_{m-3}$ does not contain all of $\mathcal X.$ Indeed, if $\mathcal X\subset \sigma_{m-3}$  then the polynomial $\sigma_{m-3}$ vanishes on $\mathcal X$ but not on $\sigma_m.$ Hence $\mathcal X$ is not $n$-complete in $\sigma_m$ which is a contradiction.
Next, choose a point $A\in \mathcal X\setminus \sigma_{m-3}.$ We have that $A\in\mathcal Y.$ Consider the fundamental polynomial $p^\star_{A, \mathcal Y}.$
Finally, notice that $p:=\sigma_{m-3}p^\star_{A, \mathcal Y}$ is a fundamental polynomial of $A$ in the set $\mathcal X$ of degree $\kappa=m+n-3.$ Hence, the set $\mathcal X$ is not essentially $\kappa$-dependent.
\end{proof}
\begin{theorem}\label{2}
Assume that $m\le n+2.$ Then we have that any set of points $\mathcal X,$ with $\#\mathcal X\le m(\kappa+3-m)-1= mn-1,$ in an irreducible curve $\sigma_m,$   is $\kappa$-independent.
\end{theorem}
\begin{proof} The cases $m=1$ and $m=2$ are evident. Suppose that $m\ge 3.$ Let us add a point $A\in\sigma_m\setminus\mathcal X$ to $\mathcal X.$ If the resulted set ${\mathcal Y}:={\mathcal X}\cup\{A\}$ is $\kappa$-independent
then $\mathcal X\subset \mathcal Y$ is also $\kappa$-independent and Theorem is proved. Now, suppose that $\mathcal Y$ is $\kappa$-dependent. Then, according to Theorem \ref{1}, (b), it is
not $n$-complete in $\sigma_m.$ Then we get from Corollary \ref{1'} that $\mathcal Y$ is essentially $\kappa$-dependent and ${\mathcal X}={\mathcal Y}\setminus \{A\}$ is $\kappa$-independent.
\end{proof}

\begin{theorem}\label{2'}
Assume that $\sigma_m$ is a curve of degree $m$, which is either not reducible or is reducible such that all its irreducible components are not empty with respect to a set $\mathcal X\subset\sigma_m.$ Assume also that $\#\mathcal X\le mn-1=m(\kappa+3-m)-1,$ where $m\le n+2.$ Then $\Xset$ is not essentially $\kappa$-dependent. \end{theorem}
\begin{proof}
The cases $m=1$ and $m=2$ are evident. Suppose that $m\ge 3.$ The case when $\sigma_m$ is irreducible follows immediately from Theorem \ref{2}. Now assume that $\sigma_m$ is reducible, i.e.,
$$\sigma_m=\sigma_{m_1}\cdots\sigma_{m_s},$$
 where the component $\sigma_{m_i}$ is irreducible and has degree $m_i.$

Assume, by way of contradiction, that $\mathcal X$ is essentially $\kappa$-dependent. Consider the set  ${\mathcal X}_i,\ i=1,\ldots,s,$ given in \eqref{comp'}. By the hypothesis ${\mathcal X}_i\neq  \emptyset,\ i=1,\ldots,s.$
Since $\mathcal X$ is essentially $\kappa$-dependent we get from Lemma \ref{reduc} that the set ${\mathcal X}_i$  is essentially $(\kappa-m+m_i)$-dependent.
 Next we are going to apply here Theorem \ref{2}. Note that the condition $m\le n+2$ here reduces to $m_i\le (\kappa-m+m_i)-m_i+5,$ which is satisfied, since in its turn it reduces  to $m_i\le \kappa-m+5=n+2.$ Now, we conclude from Theorem \ref{2} that $\#{\mathcal X}_i\ge m_i[(\kappa-m+m_i)-m_i+3)]=m_i(\kappa-m+3).$ From here, by summing up, we get $\#\mathcal X\ge m(\kappa-m+3)=mn,$ which is a contradiction.
\end{proof}

\begin{proposition} \label{3}Suppose that $m\le n.$ If a point set $\mathcal X,$ with $\#\mathcal X\le mn,$ is essentially $\kappa$-dependent then all the points of $\mathcal X$ lay in a curve of
degree $m$ or  $n-3.$
\end{proposition}
\begin{proof}
The cases $n=1,2,3,$ are evident. Thus assume that $n\ge 4.$ Suppose conversely that there is no curve of
degree $m$ containing all of $\mathcal X.$ Then there is an $m$-poised subset $\mathcal Y \subset \mathcal X$ of $(1/2)m(m+3)+1$ points.

Set $\mathcal Z= \mathcal X \setminus \mathcal Y.$
Next we are going to show that
\begin{equation} \label{aa}\#\mathcal Z\le \dim\Pi_{n-3}-1.\end{equation}

We have that
$\nu:=\#\mathcal Z- \dim\Pi_{n-3}+1\le mn-(1/2)m(m+3)-1-(1/2)n(n-3)-1= (1/2)m(2n-m-3)-1(1/2)m(2n-m-3)-(1/2)n(n-3)-1=-(1/2)(n-m-3)(n-m)-1.$

Now, evidently $\nu< 0$ if $n=m$ or $n\ge m+3.$ While $\nu = 0$ if $n=m+1$ or $n=m+2.$
Thus \eqref{aa} is proved.

By Lemma \ref{curve} there is a curve $\sigma_{n-3}$ of degree $n-3$ passing through all the points of $\mathcal Z.$
We claim that $\mathcal X\subset \sigma_{n-3}.$ Suppose by contradiction that there is a point $A\in \mathcal X\setminus \sigma_{n-3}.$
Recall that the set $\mathcal Y$ is $m$-poised and consider the $m$-fundamental polynomial $p^\star_{A, \mathcal Y}.$
Now, notice that $p:=\sigma_{n-3}p^\star_{A, \mathcal Y}$ is a $\kappa$-fundamental polynomial of the point $A$ in the set $\mathcal X,$ which is a contradiction.
\end{proof}

\begin{proposition} \label{4}  Suppose that $m\le n.$ If a set $\mathcal X$ of $mn$ points is essentially $\kappa$-dependent then all the points of $\mathcal X$ lay in a curve of
degree $m.$
\end{proposition}
\begin{proof}
Assume by the way of contradiction that $\mathcal X$ does not lay in a curve of
degree $m.$

First let us prove that there is a number $m_0>m$ such that

1) $m_0\le \frac{\kappa+3}{2},$  i.e., $m_0\le n_0:=\kappa+3-m_0,$

2) all the points of $\mathcal X$ lay in a curve of
degree $m_0,$

3) no curve of degree less than $m_0$ contains all of $\mathcal X.$

To this end let us apply Theorem \ref{3} for $\mathcal X$  and $m=m'=[\frac{\kappa+3}{2}].$
If $m'=\frac{\kappa+3}{2}$ then we get that $\mathcal X$ lies in a curve $\sigma_{m'}$ of degree $m'$ or in a curve $\sigma_{n'-3}$ of degree $n'-3,$
 where $n':=\kappa+3-m'=m',$ and conclude that  $\mathcal X$ lies in a curve $\sigma_{m'}.$

If $m'=\frac{\kappa+3}{2}-\frac{1}{2}$ then we get that $\mathcal X$ lies in a curve $\sigma_{m'}$ of degree $m'$ or in a curve $\sigma_{n'-3}$ of degree $n'-3=m'-2,$
and again conclude that  $\mathcal X$ lies in a curve $\sigma_{m'}.$

In both cases $m'\le \frac{\kappa+3}{2},$ so we have that $\mathcal X$ lies in a curve $\sigma_{m'},$ where $m'\le n'.$

Now denote by $m_0$ the minimal possible $m'$ with above described property and $\sigma_{m_0}$ be the corresponding curve of degree $m_0.$
Then it is easily seen that $m_0>m$ and the above conditions 1), 2) and 3)
are satisfied.

Let us verify that $mn\le m_0n_0-1.$ For this end consider the parabola $y=x(\kappa+3-x).$
Now it is easily seen that
\begin{equation}\label{aaa}mn=m(\kappa-m+3)< m_0(\kappa-m_0+3),
\end{equation}
since we have $y(m)=y(n)$ and $m<m_0<n.$

Next, suppose first that the curve $\sigma_{m_0}$ is irreducible. In view of \eqref{aaa} we conclude from Theorem \ref{2} that the set $\mathcal X$ is $\kappa$- independent, which is a contradiction. Note that here $m_0\le n_0.$

Finally, suppose that  $\sigma_{m_0}$ is a reducible curve: $\sigma_{m_0}=\sigma_{m_1}\cdots \sigma_{m_s},$ where the component $\sigma_{m_i}$ has degree $m_i,$ and is irreducible.
In view of the above condition 3) no component is empty with respect to the point set $\mathcal X.$
Now by Theorem \ref{2'} we get that $\mathcal X$ is not essentially $\kappa$-dependent, which is a contradiction.
\end{proof}

\begin{remark} \label{5} Suppose that $m\le n$ and a set $\mathcal X$ of $mn$ points is essentially $\kappa$-dependent.  Suppose also that no curve of degree less than $m$ contains all of $\mathcal X.$ Let $\sigma_m$ be the curve of degree $m$ from Proposition \ref{4} containing all of $\mathcal X.$ Then if the curve is reducible: $\sigma_m=\sigma_{m_1}\cdots \sigma_{m_s},$ where each component $\sigma_{m_i}$ has degree $m_i$ and is irreducible, then no point of $\mathcal X$ is an intersection point of the components and each component $\sigma_{m_i}$ contains exactly $m_i(\kappa-m+3)$ points of $\mathcal X$ which are essentially $(\kappa-m+m_i)$-dependent.
\end{remark}
 Indeed, the proof coincides with the last paragraph of the proof of Proposition \ref{4}, with $m_0$ replaced by $m.$

\begin{theorem}\label{6}
Given a set $\mathcal X, \#{\mathcal X}=mn,\ m\le n,$ satisfying the following conditions:\\
a) Any plane curve of degree $\kappa=m+n-3$ containing all but one point of $\mathcal X$, contains
all of $\mathcal X,$\\
b) No curve of degree less than $m$ contains all of $\mathcal X.$\\
Then $\mathcal X$ is the set of intersection points of some two plain curves $\sigma_m$ and $\sigma_n$
of degrees $m$ and $n,$ respectively:
\begin{equation}\label{bb}\mathcal X=  \sigma_m\cap\sigma_n.\end{equation}
\end{theorem}

\begin{proof}  The cases $m=1,2,3$ were considered earlier. Hence, suppose that $m\ge 4.$ We have from Proposition \ref{4} that  all the points of $\mathcal X$ lay in a curve $\sigma_m$ of
degree $m.$ Then we get from Proposition \ref{1''} that the set $\mathcal X$ is not $n$-complete in $\sigma_m.$

Thus  the set $\mathcal X$ is not $n$-complete on $\sigma_m.$ Therefore there exists a curve $\sigma_n$ of degree $n$ which vanishes on all the points of $\mathcal X$ but
does not vanish on whole curve $\sigma_m.$ It only remains to show that the curves $\sigma_m$ and $\sigma_n$ do not have a common component.
Suppose by way of contradiction that
$$\sigma_m= \sigma_l\sigma_{m-l} \ \hbox{and}\ \sigma_n= \sigma_l\sigma_{n-l},$$
where $\sigma_i$ has degree $i$ and the curves $\sigma_{m-l}, \sigma_{n-l}$ have no common component.

Denote $\mathcal Y:=\sigma_{m-l}\cap\sigma_{m-l}\cap\mathcal X.$
In view of the condition b) we have that $\mathcal Y\neq \emptyset.$ Let $A\in\mathcal Y.$
By the Cayley-Bacharach theorem we have that $A$ has a fundamental polynomial $p^\star_{A,\mathcal Y}$ of degree $m+n-2l-2$ in the set $\mathcal Y.$
Now notice that the polynomial $p=\sigma_lp^\star_{A,\mathcal Y}$ of degree $m+n-l-2\le m+n-3$ is a fundamental polynomial of $A$ in the set $\mathcal X,$
which contradicts the condition a). Therefore \eqref{bb} is proved.
\end{proof}
Now we get from Theorems \ref{main} and \ref{CB} the following
\begin{corollary}\label{7}
Given a set $\mathcal X, \#{\mathcal X}=mn,\ m\le n,$ satisfying the following conditions:\\
a) The set $\mathcal X$ is essentially $\kappa$-dependent,\\
b) The set $\mathcal X$ contains an $(m-1)$-poised subset.\\
Then for any point $A\in \mathcal X$ the point set $\mathcal X\setminus \{A\}$ is $\kappa$-independent.
\end{corollary}

{\noindent H. Hakopian, D. Voskanyan\\
Department of Informatics and Applied Mathematics\\
Yerevan State University\\
A. Manukyan St. 1\\
0025 Yerevan, Armenia \\}

\noindent E-mails - \texttt{hakop@ysu.am}, \quad
\texttt{ysudav@gmail.com}

\begin{thebibliography}{99}


\bibitem{EGH96}
D.~Eisenbud, M.~Green, and J.~Harris,
{Cayley-Bacharach theorems and conjectures}, Bull.\ Amer.\ Math.\ Soc.\ (N.S.) {\bf 33}(3) (1996) 295--324.

\bibitem{H00}
H.~Hakopian,
{On a class of Hermite interpolation problems},
Adv.\ Comput.\ Math.\ {\bf 12} (2000) 303--309.

\bibitem{HJZ09a}
H.~Hakopian, K.~Jetter, and G.~Zimmermann,
{Vandermonde matrices for intersection points of curves},
Ja\'en J.\ Approx.\ {\bf 1} (2009) 67--81.

\bibitem{HJZ09b}
H.~Hakopian, K.~Jetter, and G.~Zimmermann,
{A new proof of the Gasca-Maeztu conjecture for $n=4$},
J.\ Approx.\ Theory {\bf 159} (2009) 224--242.

\bibitem{HJZ14}
H. Hakopian, K. Jetter and G. Zimmermann, The Gasca-Maeztu
conjecture for $n=5$,  {Numer. Math.  {\bf 127} (2014) 685--713.}

\bibitem{HM12}
H.~Hakopian and A.~Malinyan,
{Characterization of $n$-independent sets with no more than $3n$ points},
Ja\'en J.\ Approx. {\bf 4} (2012) 119--134.





\bibitem{HM13}
H.~Hakopian and A.~Malinyan,
{On $n$-independent sets located on quartics}
Proceedings of YSU, Phys. and Math. Sci.  {\bf 1} (2013) 6--12.

\bibitem{R11}
L.~Rafayelyan,
{Poised nodes set constructions on plane algebraic curves},
East J.\ Approx.\ {\bf 17} (2011) 285--298.

\bibitem{S}
Severi, F.: Vorlesungen \"uber Algebraische Geometrie, Teubner, Berlin (1921). (Translation into
German - E. L\"offler)

\end{thebibliography}
\end{document}